\documentclass[3p,twocolumn]{elsarticle}
\usepackage{amsmath,amssymb,amsfonts}
\usepackage{mathtools}    
\newcommand*{\id}{{\mathrm{id}}}

\newcommand*{\R}{{\mathbb R}}                                                 

\newcommand*{\CC}{{\mathbb{C}}}

\newcommand*{\Vb}{{\mathbb V}}

\newcommand*{\pa}{{\partial}}

\newcommand*{\rd}{{\mathrm d}}

\newcommand*{\gf}{{\mathfrak g}}  
\newcommand*{\qf}{{\mathfrak q}} 
\newcommand*{\pf}{{\mathfrak p}} 
\newcommand*{\af}{{\mathfrak a}}
\newcommand*{\ef}{{\mathfrak e}}
\newcommand*{\ff}{{\mathfrak f}}
\newcommand*{\hf}{{\mathfrak h}}
\newcommand*{\of}{{\mathfrak o}}
\newcommand*{\lf}{{\mathfrak l}}

\newcommand*{\conv}{\ast} 

\newcommand*{\bs}{\boldsymbol}

\newtheorem{theorem}{Theorem}
\newtheorem{lemma}[theorem]{Lemma}
\newtheorem{corollary}[theorem]{Corollary}
\newdefinition{definition}{Definition}
\newdefinition{prescription}{Prescription}
\newdefinition{remark}{Remark}
\newdefinition{example}{Example}
\newproof{proof}{Proof}

\journal{a Mathematics Journal}

\begin{document}
\begin{frontmatter}
\title{Applications of Grassmannian flows to coagulation equations}
\author{Anastasia Doikou}    
\ead{A.Doikou@hw.ac.uk}
\author{Simon~J.A.~Malham\corref{cor1}\fnref{fn1}}    
\ead{S.J.A.Malham@hw.ac.uk} 
\author{Ioannis Stylianidis\fnref{fn2}}    
\ead{is11@hw.ac.uk}
\author{Anke Wiese}    
\ead{A.Wiese@hw.ac.uk}
\cortext[cor1]{Corresponding author, 28th March 2023}
\fntext[fn1]{SJAM was supported by an EPSRC Mathematical Sciences Small Grant EP/X018784/1}
\fntext[fn2]{IS was supported by an EPSRC DTA Scholarship}
\affiliation{organisation={Maxwell Institute for Mathematical Sciences,        
and School of Mathematical and Computer Sciences},   
addressline={Heriot-Watt University}, postcode={EH14 4AS}, city={Edinburgh}, country={UK}}

\begin{abstract}
We demonstrate how many classes of Smoluchowski-type coagulation models can be
realised as multiplicative Grassmannian flows and are therefore linearisable, and thus integrable in this sense.
First, we prove that a general Smoluchowski-type equation with a constant frequency kernel,
that encompasses a large class of such models, is realisable as a multiplicative Grassmannian flow.
Second, we establish that several other related constant kernel models
can also be realised as such. These include: the Gallay--Mielke coarsening model; the Derrida--Retaux
depinning transition model and a general mutliple merger coagulation model.
Third, we show how the additive and multiplicative frequency kernel cases can be
realised as rank-one analytic Grassmannian flows.  
\end{abstract}
\end{frontmatter}

\section{Introduction}\label{sec:intro}
Our goal herein is to demonstrate how many classes of partial differential equations
with nonlocal coagulation interaction nonlinearities are integrable as Grassmannian flows.
We think of the Grassmann manifold as the set of all compatible graphs of linear maps.
First, we establish that a general Smoluchowski-type equation with constant frequency kernel is a multiplicative Grassmannian flow.
Hence it is linearisable and integrable in this sense. 
By this we mean that we can generate the solution by solving the corresponding linearised version of the 
general Smoluchowski-type equation, in fact the Laplace transform form of this linear equation,
and then solving a linear algebraic relation for the Laplace transform of the solution.
Second, we show that several related applications that do not quite fit into
the general Smoluchowski-type class just mentioned, are also realisable as multiplicative Grassmannian flows.
The applications include the Gallay--Mielke coarsening model, the Derrida--Retaux depinning transition model and a general mutliple merger coagulation model.
Third, we consider the classical additive and multiplicative frequency kernel cases that correspond to the
inviscid Burgers equation in Laplace transform space. Assuming analytic initial data, we show that the inviscid Burgers equation can be realised as
a rank-one analytic Grassmannian flow and is thus linearisable in this sense---up until the gelation time in the multiplicative case.
We achieve this by pulling back the exponential series for the analytic solution to the corresponding sequence space of Taylor coefficients.
We equivalence by the natural group of transformations based on the Fa\`a di Bruno algebra of compositions of functions (and Lagrange Inversion Theorem).
This is also a linear procedure, the problem boils down to solving a linear algebraic system of equations equivalent to a simple Fredholm equation
on the sequence space mentioned.

Let us give some brief context. The classical Smoluchowski coagulation equation has the form,
\begin{align}\label{eq:Smoluchowski}
  \pa_tg(x;t)=&\;\tfrac12\int_0^x K(x-y,y)g(x-y;t)g(y;t)\,\rd y\nonumber\\
  &\;-g(x;t)\int_0^\infty K(x,y)g(y;t)\,\rd y,
\end{align}
where $g=g(x;t)$ denotes the density of molecular clusters of mass $x$, and
the form of the frequency kernel $K=K(x,y)$ depends on the application at hand.
Three specific forms that are of particular interest are the constant, $K=1$, additive, $K=x+y$, and multiplicative, $K=xy$, kernel cases.
Smoluchowski's equation can be solved explicitly for these cases, though in the multiplicative kernel case only locally in time, up until the time of gelation.
In the constant kernel case, in Laplace transform space, the solution satisfies a scalar Riccati equation. 
In the additive and multiplicative kernel cases, the appropriate de-singularised Laplace transform of the solution satisfies an inviscid Burgers flow.
See Menon and Pego~\cite[Sec.~2]{MP} for more details. Classically, Smoluchowski's coagulation equation is a model for 
polymerisation, aerosols, clouds/smog, clustering of stars and galaxies as well as schooling and flocking;
see for example Aldous~\cite{Aldous}. Recently it has been used as a model for genealogy, see Lambert and Schertzer~\cite{LS};
coarsening, see Gallay and Mielke~\cite{GM}; nanostructures on substrates including ripening or `island coarsening', see Stoldt \textit{et al.} \cite{SJTCE},
growth and morphology of nanocrystals at atomic and nanoparticle scales, see Woehl \textit{et al.} \cite{WPEA} and Kaganer \textit{et al.\/} \cite{KPS},
epitaxial $\text{Si}_{1-\alpha}\text{Ge}_{\alpha}/\text{Si}(001)$ islands, see Budiman and Ruda~\cite{BR},
growth of graphene on Ir(111), see Coraux \textit{et al.\/} \cite{CNDEBWBMGPM}; and gold nanoparticles on a silicon substrate, see Winkler \textit{et al.} \cite{WWLGF};
blood clotting including  coagulation and formation of fibrin gel, see Guy \textit{et al.\/} \cite{GFK} and Rouleau formation where 
blood cells aggregate to form cylindrical clusters or `Rouleaux', see Samsel and Perelson~\cite{SPI,SPII} and polymer growth of proteins in biopharmaceuticals,
see Galina \textit{et al.\/} \cite{GLK} or Zidar \textit{et al.\/} \cite{ZKR}.

Our first main goal herein is to derive the solution to a general Smoluchowski-type equation with constant kernel $K=1$,
which has the form,
\begin{align}\label{eq:genSmoluchowski}
  \pa_tg(&x;t)\nonumber\\
  =&\;\int_0^xg(x-y;t)b(\pa_y)g(y;t)\,\rd y\nonumber\\
  &\;-g(x;t)\int_0^\infty g(y;t)\,\rd y+d(\pa_x)g(x;t) \nonumber\\
  &\;-\int_0^x\int_0^yg(z;t)b_0(y-z)\,\rd z\,g(x-y;t)\,\rd y\nonumber\\
  &\;-\int_0^xg(x-y;t)a(y;t)\,\rd y.
\end{align}
Here $d=d(\pa_x)$ and $b=b(\pa_x)$ are partial differential operators, $a=a(y;t)$ is a given function and $D_0$, $B_0$ and $\beta$ are constants.
Depending on the form of $d=d(\pa_x)$, the term `$d(\pa_x)g(x;t)$' generates either a diffusive or dispersive effect.
We think of this equation as a general version of the constant kernel Smoluchowski equation 
with a similar non-local nonlinearity and additional terms, including a diffusion term.
However, with the inclusion of the diffusion/dispersion term, we no longer strictly associate the
equation with a coagulation process and its solution $g=g(x;t)$ with the density of molecular clusters of mass $x$.
In particular for example, we do not expect there to be dissipation of the form `$d(\pa_x)g(x;t)$' of mass clusters 
in the original coagulation context. However, the other local and nonlocal terms shown could be interpreted as special fragmentation terms.
This is the general Smoluchowski-type equation we show can be linearised and constitutes a multiplicative Grassmannian flow.
By this we mean the following.
For convenience, we transport the equation to the Laplace transform setting with the solution form $\gf=\gf(s;t)$
representing the Laplace transform of $g=g(x;t)$. From the linearised version of~\eqref{eq:genSmoluchowski}
in the Laplace transform space, which essentially means we ignore the first, second and forth terms on the right in~\eqref{eq:genSmoluchowski},
we can generate a pair of linear evolution equations for the auxillary Laplace transform fields $\pf=\pf(s;t)$ and $\qf=\qf(s;t)$.
See Prescription~~\ref{def:lineareqns} for the explicit form these linear equations.
In Section~\ref{sec:Smoluchowski} we show that, for some $T>0$, the solution $\gf=\gf(s;t)$ is generated by solving
the linear multiplicative algebraic relation $\pf=\gf(1+\qf)$ at any time $t\in[0,T]$. Indeed, we can determine an explicit solution form. 
We also outline more explicitly therein why this represents a multiplicative Grassmannian flow.

Our second main goal herein is to show that five example models, that do not exactly fit the format
of the general Smoluchowski-type equation~\eqref{eq:genSmoluchowski}, can also be linearised and constitute multiplicative Grassmannian flows
(one we only partially solve). The models are disparate, but they do involve convolution-type nonlocal nonlinearities and thus we expected them
to be amenable to the linearisation approach we propose. The five models, which we present in Section~\ref{sec:examples}, are:
the coarsening model of Gallay and Mielke~\cite{GM}; the \emph{toy} depinning transition model of Derrida and Retaux~\cite{DerridaRetaux}---for
the transition of a DNA molecule into two single strands or a line from a substrate;
the geneology model of Lambert and Schertzer~\cite{LS}---this is the model we only partially solve;
the mutliple merger coagulation model of Iyer \textit{et al.\/} \cite{ILP} which involves higher degree nonlinearities;
and finally a pre-Laplace viscous Burgers model which can also be solved explicitly.

Our third main goal herein is to show that the Smoluchowski equation~\eqref{eq:Smoluchowski} in the 
additive and multiplicative kernel cases can also be represented as a Grassmannian flow.
For these cases, in the Laplace transform setting, the solution $\gf=\gf(s;t)$ satisfies the inviscid Burgers equation,
respectively, with and without linear damping; see Menon and Pego~\cite{MP}. These equations can be solved by characteristics.
Ultimately this involves inverting the charactistics flow map to determine their initial labels. For analytic data,
we can represent the characteristics flow as an exponential power series, at each time (up to gelation in the multiplicative case).
Inversion involves determining the coefficients of the inverse flow map from the flow map, which boils down to solving a (infinite) linear algebraic system of equations.
The Fa\`a di Bruno formula and Lagrange Inversion Theorem are key components in this procedure; see Figueroa \text{et al.\/} \cite{FG-BV} and Gessel~\cite{Gessel}.
The end result is that the inviscid Burgers equation underlying both cases, is linearisable as a rank-one analytic Grassmannian flow.

In summary, the new results herein are, we:
\begin{enumerate}
\item[(i)] Prove the solution flow of a general Smolu-chowski-type equation
with a constant frequency kernel, that encompasses a large class of models with nonlocal convolution nonlinearity,
is realisable as a multiplicative Grassmannian flow. The flow is thus linearisable and integrable in this sense;
\item[(ii)] Use the Grassmannian flow techniques from (i) to prove that five related constant kernel models, which include
the Gallay--Mielke coarsening, Derrida--Retaux depinning transition and general mutliple merger coagulation models,
are also integrable in an analogous sense;
\item[(iii)] Show that the invsicid Burgers equation, which corresponds to the classical additive and multiplicative frequency kernel
cases, is linearisable as a rank-one analytic Grassmannian flow.
\end{enumerate}

Our paper is structured as follows. In Section~\ref{sec:Smoluchowski} we show the general Smoluchowski-type equation~\eqref{eq:genSmoluchowski}
can be realised as a multiplicative Grassmannian flow. In Section~\ref{sec:examples} we demonstrate that 
the coarsening and multiple merger models mentioned are also multiplicative Grassmannian flows.
In Section~\ref{sec:FredholmGrassmannians}, we show that the additive and multiplicative kernel cases correspond to
rank-one analytic Grassmannian flows. Finally in Section~\ref{sec:discussion} we briefly discuss future directions.

\section{General Smoluchowski-type equation}\label{sec:Smoluchowski}
We extend the programme for Grassmannian flows we developed in Beck \textit{et al.\/} \cite{BDMS1,BDMS2}
to general Smoluchowski-type equations. Herein, we first explore the Smoluchowski equation~\eqref{eq:Smoluchowski} with constant frequency kernel $K=1$ in detail,
demonstrating that it is a multiplicative Grassmannian flow. 
We then focus on our main result and prove that the more general Smoluchowski-type equation~\eqref{eq:genSmoluchowski} is
also naturally a multiplicative Grassmannian flow.  
To begin, consider the Smoluchowski equation~\eqref{eq:Smoluchowski} in the constant frequency kernel $K=1$ case.
We note that the total mass $\int_{[0,\infty)}xg(x;t)\,\rd x$ of clusters is constant. This can be established by multiplying the
equation above by $x$, integrating over $x\in[0,\infty)$, swapping the order of integration in the first term and making a simple change of variables.
(This also applies to the additive kernel case, and to the multiplicative kernel case up to the time of gelation.)
Further, the evolution of the total number of clusters in Smoluchowski's coagulation equation,
$M(t)\coloneqq\int_0^\infty g(x;t)\,\rd x$, is given as follows. Integrating Smoluchowski's equation~\eqref{eq:Smoluchowski} over $x\in[0,\infty)$
and swapping the order of integration in the first term, reveals $\dot{M}=-\frac12M^2$ and so, $M(t)=2M_0/(2+tM_0)$,
where $M(0)=M_0$, with $M_0$ representing the initial total number of clusters.
For any function $g=g(x)$ such that $g\in L^1_{\mathrm{loc}}\bigl([0,\infty);\R\bigr)$, i.e.\/ it is locally integrable on $[0,\infty)$,
and which is of exponential order for large $x$, i.e.\/ $g(x)=\mathcal O(\mathrm{e}^{cx})$ as $x\to+\infty$ for some constant $c\geqslant0$,
its Laplace transform $\gf=\gf(s)$ is defined for $\mathrm{Re}(s)>c$ and given by
\begin{equation*}
\gf(s)\coloneqq\int_{0^-}^\infty \mathrm{e}^{-sx}\,g(x)\,\rd x.
\end{equation*}
Further, $\gf$ is analytic in $s\in\CC$ for $\mathrm{Re}(s)>c$.
We set $\CC_+\coloneqq\{\sigma\in\CC\colon\mathrm{Re}(\sigma)\geqslant0\}$, for convenience hereafter. 
The assumptions on $g$ above can be replaced by $g\in L^1\bigl([0,\infty);\R\bigr)$.
The inverse Laplace transform is given by the Bromwich contour integral,
\begin{equation*}
g(x)\coloneqq\int_{\gamma-\mathrm{i}\infty}^{\gamma+\mathrm{i}\infty}\mathrm{e}^{sx}\gf(s)\,\rd s,
\end{equation*}
where the constant $\gamma\geqslant0$ is chosen such that $\gamma>c$.
For both integrals, the limits shown are understood in the obvious sense.
With $M=M(t)$ determined, we observe $g=g(x;t)$ satisfies 
\begin{equation*}
\pa_tg(x;t)=\tfrac12\int_0^xg(y;t)g(x-y;t)\,\rd y-g(x;t)M(t).
\end{equation*}
The natural context for the Smoluchowski equation is Laplace transform space, see Menon and Pego~\cite{MP}.
If $g=g(x;t)$ satisfies Smoluchowski's equation with $K=1$,
then its Laplace transform $\gf=\gf(s;t)$ satisfies the Riccati equation,
\begin{equation}\label{eq:Riccati}
\pa_t\gf=\tfrac12\gf^2-M\gf.
\end{equation} 
The natural prescription for this Riccati equation is given by the following linear system. 
Assume we are given arbitrary data $\gf_0=\gf_0(s)$ and $\gf(s;0)=\gf_0(s)$.
\begin{prescription}[Smoluchowski $K=1$]\label{prescription:constantkernel}
Suppose the functions $\qf=\qf(s;t)$, $\pf=\pf(s;t)$ and $\gf=\gf(s;t)$
satisfy $\qf(s;0)=0$ and $\pf(s;0)=\gf_0(s)$ and the linear system of equations, 
\begin{equation*}
\pa_t\pf=-M\pf,\quad \pa_t\qf=-\tfrac12\pf,\quad \pf=\gf(1+\qf).
\end{equation*}
\end{prescription}
A straightforward calculation reveals $\gf$ satisfies the Riccati equation above
provided that $1+\qf\neq0$. Since $\qf(s;0)=0$, this is guaranteed for $t\in[0,T]$ for some $T>0$.
We show that we can take $T=\infty$, presently, in Example~\ref{example:explicitsolution}.
Naturally there is a corresponding linear prescription in mass cluster space.
\begin{example}\label{example:explicitsolution}
We use the linear system in Prescription~\ref{prescription:constantkernel}
to derive the explicit solution to the Smoluchowski equation in the constant kernel case
which, for example, can be found in Scott~\cite{Scott}. We observe that, by direct integration
of the formula for the total number of clusters $M=M(t)$ above: $\int_0^tM(\tau)\,\rd\tau=\log\bigl((2+tM_0)/2\bigr)^2$.
Solving the equation for $\pf=\pf(s;t)$ in Prescription~\ref{prescription:constantkernel}, we find
$\pf=\exp\bigl(-\int_0^tM(\tau)\,\rd\tau\bigr)\gf_0(s)=4\gf_0(s)/(2+tM_0)^2$,
where we have used that $\pf(s,0)=\gf_0(s)$. Integrating the equation for $\qf=\qf(s;t)$
therein and using that $\qf(s;0)=0$, we observe
\begin{equation*}
  1+\qf(s;t)=1-\frac{\gf_0(s)}{M_0}+\frac{2\gf_0(s)}{M_0(2+tM_0)}.
\end{equation*}
Since $\gf(s;t)=\pf(s;t)/\bigl(1+\qf(s;t)\bigr)$, we observe that in fact,
$\gf(s;t)=4\gf_0(s)/(2+tM_0)(2+tM_0-t\gf_0(s))$.
The solution $g=g(x;t)$ to the Smoluchowski equation in the constant kernel $K=1$ case, is thus given by
inverse Laplace transform integral of this expression for $\gf=\gf(s;t)$.
Once we have accounted for the normalization of $g=g(x;t)$ by $M_0$, this exactly matches the solution
to the Smoluchowski equation in the constant kernel $K=1$ case in Scott~\cite[eq.~(4.3)]{Scott}.
Further note that $1+\qf(s;t)=0$ if and only if $t=2/(\gf_0(s)-M_0)$.
From the definition of $\gf_0=\gf_0(s)$ as the Laplace transform of the initial data $g_0=g_0(x)$,
we observe that for all $s\in\CC_+$, $|\gf_0(s)|\leqslant\gf_0(0)\equiv M_0$.
In particular, as expected, by taking the large $s$ limit we observe that $1+\qf(s;t)$ is non-zero and there is
no blow-up in the solution $g=g(x;t)$ for $t\in(-2/M_0,\infty)$.
\end{example}
\begin{remark}[Measure-valued solutions]\label{rmk:Smoluchowskirigorous}
Menon and Pego~\cite{MP} establish the existence and uniqueness of weak solutions 
to Smoluchowski's equation in the $K=1$ case, in the sense of positive Radon measures on $(0,\infty)$.
We can invoke the results they establish here to affirm that$^\ddag$, as a simple addendum, the 
linear equations given in Prescription~\ref{prescription:constantkernel} determine such solutions.
Briefly, the \emph{desingularised} Laplace transform,
\begin{equation}\label{eq:desingLT} 
\gf(s;t)\coloneqq\int_0^\infty(1-\mathrm{e}^{-sx})\,\nu_t(\rd x), 
\end{equation}
of a time-dependent, positive measure-valued solution $\nu_t$ to the constant kernel Smoluchowski equation,
satisfies $\pa_t\gf=-\frac12\gf^2$. This is a renormalised version of the Riccati equation~\eqref{eq:Riccati}, 
with the total mass term involving $M=M(t)$ `knocked out', and a sign change on the right due to the definition of~\eqref{eq:desingLT}.
The solution to this Riccati equation for $\gf=\gf(s;t)$ is the solution given in Example~\ref{example:explicitsolution},
with $M_0$ set equal to zero and a sign change for $t$.
In other words, $\gf=2\gf_0(s)/(2-t\gf_0(s))$. Menon and Pego~\cite{MP} show that if $\nu_0$
is an initial positive radon measure, then this solution for $\gf=\gf(s;t)$ determines a weakly continuous map
$t\mapsto\nu_t$ into the set of positive Radon measures, for all times $t\in[0,\infty)$, with $\nu_t$ the
aforementioned weak solution of constant kernel Smoluchowski equation. Since Prescription~\ref{prescription:constantkernel},
with $M=M(t)$ set to zero and $\pa_t\qf=\frac12\pf$, determines $\gf$, and $\gf=\gf(s;t)$ determines the measure valued solutions just mentioned,
our statement `$\ddag$' just above follows.
\end{remark}

The linear equations for $\qf$ and $\pf$, and the linear multiplicative relation $\pf=\gf(1+\qf)$, constitute 
what we call a \emph{multiplicative} Grassmannian flow. Let us now outline what we mean by this.
Indeed, in Laplace transform space, suppose $Q=Q(t)$ and $P=P(t)$ are the time-dependent multiplicative linear operators given by,
\begin{align*}
  Q(t)&\colon\varphi(s)\mapsto\qf(s;t)\varphi(s),\\
  P(t)&\colon\varphi(s)\mapsto\pf(s;t)\varphi(s).
\end{align*}
Such muliplicative operators are bounded linear operators on a vector space $\Vb$ provided
$\qf=\qf(\,\cdot\,;t)$ and $\pf=\pf(\,\cdot\,;t)$ are bounded functions since,
for example, $\|\qf(t)\varphi\|_{\Vb}\leqslant\|\qf(t)\|_\infty\|\varphi\|_{\Vb}$. 
In particular, provided the essential infinum of $1+\qf(\,\cdot\,;t)$ is strictly positive,
then we know $(1+\qf(\,\cdot\,;t))^{-1}$, and thus $(\id+Q(t))^{-1}$, exist.
In Doikou \textit{et al.\/} \cite{DMSW-integrable} we outline the structure of Grassmannian flows.
We refer the reader there for more details.
The multiplicative operator setting we have here, simplifies the flow structure considerably.
Such multiplicative operators correspond to general diagonal operators.
Thus, for example, if we were to represent $Q=Q(t)$ by an integral kernel---suppose that $\Vb=L^2([0,\infty);\CC)$
and $Q$ is a Hilbert--Schmidt operator on $\Vb$---then the integral kernel would be $\qf(s;t)\delta(s-\sigma)$. 
Herein, we can, and do, proceed at the multiplicative function level---we find explicit solution forms for
$\qf=\qf(s;t)$ and $\pf=\pf(s;t)$ in our main examples in this section and Section~\ref{sec:examples}.
Thus, here, we have a linear (Stiefel) flow parameterised by $(1+\qf(s;t),\pf(s;t))^{\mathrm{T}}$.
Assuming $(1+\qf(s;t))^{-1}$ exists, then multiplication by $(1+\qf(s;t))^{-1}$
projects the flow $(1+\qf(s;t),\pf(s;t))^{\mathrm{T}}$ onto $(1,\gf(s;t))^{\mathrm{T}}$ where
$\gf(s;t)=\pf(s;t)(1+\qf(s;t))^{-1}$. The form $(1,\gf)^{\mathrm{T}}$ represents a graph of a given linear map $\gf$.
The flow $(1,\gf(s;t))^{\mathrm{T}}$ represents a flow in the coordinate patch corresponding to the top cell of the multiplicative Grassmannian. 
If $1+\qf(s;t)$ is zero, we choose a different coordinate chart. Indeed we multiply $(1+\qf(s;t),\pf(s;t))^{\mathrm{T}}$ by $(\pf(s;t))^{-1}$,
assuming $1+\qf(s;t)$ and $\pf(s;t)$ cannot be zero at the same time, which is true for all the examples we consider herein.
This projects the flow onto the chart $(\gf^\prime(s;t),1)^{\mathrm{T}}$ where $\gf^\prime(s;t)=(1+\qf(s;t))(\pf(s;t))^{-1}$.
We think of the Grassmannian as the set of all such compatible graphs of linear maps; again see Doikou \textit{et al.\/} \cite{DMSW-integrable}.
In our applications in this section and in Section~\ref{sec:examples}, we restrict ourselves to the top cell graph $(1,\gf(s;t))^{\mathrm{T}}$.
However, if $1+\qf(s;t)$ becomes zero so blow-up in $\gf=\gf(s;t)$ occurs, then in principle we can continue the solution
via the coordinate patch $(\gf^\prime(s;t),1)^{\mathrm{T}}$ (though we do not pursue this here).

Our main goal in this section is to derive the solution to the general Smoluchowski-type equation~\eqref{eq:genSmoluchowski}. 
We achieve this via a Grassmannian flow using a prescription for a linear system of equations, very similar to that in Prescription~\ref{prescription:constantkernel}. 
For some $n\in\mathbb N$, we define the operators $d=d(\pa_x)$ and $b=b(\pa_x)$ by, $d(\pa_x)\coloneqq-D_0-d_0\pa_x^{n}$ and $b(\pa_x)\coloneqq B_0+\beta\pa_x^m$, 
where $D_0>0$, $B_0\in(0,1)$ and $\beta\in\R$ are constants, and $m$ is a positive integer such that $m\leqslant n$.
Further we assume $d_0>0$, unless $n=2(2k-1)$ for $k\in\mathbb N$, in which case we assume $d_0<0$.  
The form of $d=d(\pa)$ assumed, generates either a diffusive or dispersive effect in the equation we now present. 
We define the spatial convolution product `$\conv$' of two functions $f=f(x;t)$
and $g=g(x;t)$ on $[0,\infty)$ by, 
\begin{equation*}
(f\conv g)(x;t)\coloneqq\int_0^xf(x-y;t)g(y;t)\,\rd y.
\end{equation*}
\begin{definition}[General equation]\label{def:gSmoluchowski}
The general Smoluchowski-type equation~\eqref{eq:genSmoluchowski} for $g=g(x;t)$, with $\pa=\pa_x$, has the following form:
\begin{align}
  \pa_tg=&\;B_0g\conv g+\beta g\conv (\pa^mg)-(M+D_0)g-d_0\pa^{n}g\nonumber\\
  &\;-g\conv b_0\conv g-g\conv a, \label{eq:origgenSmoltype}
\end{align} 
for $(x,t)\in[0,\infty)^2$, together with the initial condition $g(x;0)=g_0(x)$ for a given function $g_0(x)\geqslant0$ which
is strictly positive on some finite subinterval of $[0,\infty)$, and the boundary conditions $\pa_x^\ell g(0;t)=0$, for $\ell=0,1,2,\ldots,n-1$.
Note, all the boundary conditions are fixed at $x=0$. In the equation, $a$ is a given smooth, non-negative, integrable function on $[0,\infty)^2$
and $b_0$ is a smooth, non-negative, integrable function of $x\in[0,\infty)$. As before, $M=M(t)$ is the integral of $g=g(x,t)$ over $x\in[0,\infty)$.
\end{definition}

By integrating the Smoluchowski-type evolution equation above over $x\in[0,\infty)$,
and using that $\int_0^\infty(f\conv g)(x)\,\rd x=\bigl(\int_0^\infty f(x)\,\rd x\bigr)\bigl(\int_0^\infty g(y)\,\rd y\bigr)$ 
for any two functions $f$ and $g$ on $[0,\infty)$,
and the boundary conditions at $x=0$, as well as assuming $g$ decays as $x\to\infty$, we find, $\dot{M}=-(D_0+\overline{a})M+(B_0-1-\overline{b}_0)M^2$,
where $\overline{a}(t)\coloneqq\int_0^\infty a(x;t)\rd x$ and $\overline{b}_0\coloneqq\int_0^\infty b_0(x)\rd x$. 
We can solve this Riccati ordinary differential equation for $M=M(t)$ by the standard linearisation
approach. We suppose the functions $Q=Q(t)$, $P=P(t)$ and $M=M(t)$ satisfy $P(0)=M_0$ and $Q(0)=1$ and the system of linear equations:
$\dot Q=(1+\overline{b}_0-B_0)P$, $\dot P=-(D_0+\overline{a})P$ and $P=MQ$.
Solving the equation for $P=P(t)$ we find, $P(t)=\exp\bigl(-\int_0^t(D_0+\overline{a}(\tau))\,\rd\tau\bigr)M_0$.
Substituting this into the equation for $Q=Q(t)$, we find, $Q(t)=1+(1+\overline{b}_0-B_0)\int_0^tP(\tau)\,\rd\tau$.
Hence we observe that, $M(t)=(1+\overline{b}_0-B_0)^{-1}\pa_t\log\bigl(1+(1+\overline{b}_0-B_0)\int_0^tP(\tau)\,\rd\tau\bigr).$
Naturally we can substitute for $P=P(t)$, from the solution above, into this last formula to obtain an explicit solution for $M=M(t)$. 
Since $g_0(x)\geqslant0$, but strictly positive on some finite subinterval of $[0,\infty)$, we know $M_0>0$.
Since we assumed $a(x;t)\geqslant0$ and $b_0(x)\geqslant0$,
we know $\overline{a}(t)\geqslant0$ and $\overline{b}_0\geqslant0$. Thus since $D_0>0$ and $B_0\in(0,1)$,
we deduce that $D_0+\overline{a}(t)>0$ and $1+\overline{b}_0-B_0>0$. This means that $M=M(t)$ strictly
decreases from its initial value $M_0>0$ as time progresses. Furthermore $M=M(t)$ remains positive
throughout, though this may not be true for $g=g(x,t)$ itself. Thus, this a-priori estimate has
established that $M=M(t)$ is a bounded positive quantity on $[0,\infty)$.

Our goal now is to establish the solution to the general Smoluchowski-type equation given above.
To this end we define the class of Laplace Transform solutions. We assume all the properties
on $D_0$, $d_0$, $B_0$, $\beta$, $g_0$, $b_0$ and $a$ outlined above, hereafter.
\begin{definition}[Laplace transform solutions]\label{def:LTeqn}
We say $g=g(x;t)$ is a Laplace transform solution to the general Smoluchowski-type equation given above,
if the function $\gf=\gf(s;t)$ satisfies, 
\begin{equation*}
  \pa_t\gf=(B_0-\beta s^m-\mathfrak b_0)\gf^2-(M+\mathfrak a+D_0+d_0s^{n})\gf,
\end{equation*}
and the initial condition $\gf(s;0)=\gf_0(s)$. In these equations, $\gf_0=\gf_0(s)$, $\mathfrak b_0=\mathfrak b_0(s)$
and $\mathfrak a=\mathfrak a(s;t)$ are the respective Laplace transforms of $g_0(x)$, $b_0(x)$ and $a(x,t)$.
We assume they all exist, and are thus analytic, in $\CC_+$. Note that $M(t)\equiv\gf(0;t)$.
\end{definition}
\begin{remark}
If we take the limit $s\to0$ in the Laplace transform equation in Definition~\ref{def:LTeqn},
then we observe that, since $M(t)\equiv\gf(0;t)$, $\overline{a}(t)\equiv\mathfrak a(0;t)$ and $\overline{b}_0\equiv\mathfrak b_0(0)$,
the evolution equation for $\gf(0;t)$ is exactly the same as that for $M=M(t)$ above.
Hence for such Laplace transform solutions, we know that $M(t)\equiv\gf(0;t)$ satisfies
the properties we derived a-priori above, and in particular, has the solution formula shown.
\end{remark}
There is a simple prescription of linear equations for the equations above as follows.
\begin{prescription}[Gen.\/ Smoluchowski-type]\label{def:lineareqns}
Suppose the functions $\qf=\qf(s;t)$, $\pf=\pf(s;t)$ and $\gf=\gf(s;t)$
satisfy  $\qf(s;0)=0$ and $\pf(s;0)=\gf_0(s)$ and the system of linear equations:
\begin{align*}
\pa_t\pf&=-(M+\mathfrak a+D_0+d_0s^{n})\pf,\\
\pa_t\qf&=(\mathfrak b_0+\beta s^m-B_0)\pf,\\
\pf&=\gf(1+\qf).
\end{align*}
\end{prescription}
\begin{remark}
In the scalar commutative context here, there is some choice in how to distribute
the coefficients, of the linear terms in the evolution equation for $\gf=\gf(s;t)$ in Definition~\ref{def:LTeqn},
between the equations for $\qf=\qf(s;t)$ and $\pf=\pf(s;t)$ above. We have chosen to make the
linear evolution equation for $\qf=\qf(s;t)$ as simple as possible.
\end{remark}
We now proceed to explicitly solve the linear evolution equations in Prescription~\ref{def:lineareqns}
for $\pf=\pf(s;t)$ and then $\qf=\qf(s;t)$. The linear evolution equation for $\pf$, and then that for $\qf$,
are readily solved, and their unique forms are, 
\begin{align}
  \pf(s;t)&=\mathrm{e}^{-t(D_0+d_0s^{n})-\int_0^t(M(\tau)+\af(s;\tau))\,\rd\tau}\gf_0(s),\label{eq:pform}\\
  \qf(s;t)&=(\mathfrak b_0+\beta s^m-B_0)\int_0^t\pf(s;\tau)\,\rd\tau. \label{eq:qform}
\end{align}
Hence we find,
\begin{equation*}
\gf(s;t)=(\mathfrak b_0+\beta s^m-B_0)^{-1}\pa_t\log\bigl(1+\qf(s;t)\bigr).
\end{equation*}
We can substitute the solution above for $\pf=\pf(s;t)$ into this last expression to obtain
an explicit solution formula for $\gf=\gf(s;t)$.
Let us remark on the properties of this solution. Since $\gf_0=\gf_0(s)$ exists and is analytic in
$\CC_+$, from their forms above, we observe that both $\pf(s;t)$ and $1+\qf(s;t)$
exist for all $t\geqslant0$, are unique, and are analytic with respect to $s$ for $s\in\CC_+$.
Indeed both solutions are smooth with respect to time. Further, we also have the following.
\begin{lemma}[Reciprocal existence]\label{lemma:reciprocal}
For some $T>0$, we know for $t\in[0,T]$ we have: $|\qf(s;t)|<1$ and $1+\qf(s;t)\neq0$,
and $\exists$ a $\qf_\ast=\qf_\ast(s;t)$ such that,
\begin{equation*}
\bigl(1+\qf(s;t)\bigr)^{-1}=1+\qf_\ast(s;t),
\end{equation*}
with $\qf_\ast=\qf_\ast(s;t)$ analytic with respect to $s$ for $s\in\CC_+$
and smooth in time for $t\in[0,T]$. Naturally we have $\qf_\ast(s;0)=0$.
\end{lemma}
\begin{proof}
The form of $\qf=\qf(s;t)$ above reveals that it is the product of 
$(\mathfrak b_0+\beta s^m-B_0)$ and the time integral of $\pf=\pf(s;t)$, both
of which are analytic with respect to $s$ for $s\in\CC_+$, with
the latter smooth in time for $t\in[0,T]$. Since $\qf(s;0)=0$, for a short time at least,
$|\qf(s;t)|<1$ and $1+\qf(s;t)\neq0$. Since $|\qf(s;t)|<1$ for that short time, $(1+\qf(s;t))^{-1}$
exists and has a convergent power series expansion in $\qf=\qf(s;t)$. Since $\qf=\qf(s;t)$ is
analytic with respect to $s$ for $s\in\CC_+$ and smooth in time,
$\qf_\ast(s;t)=(1+\qf(s;t))^{-1}-1$ is also analytic for all $s\in\CC_+$, and smooth in time.  
\qed
\end{proof}
Putting these results together, we have thus established the following. 
\begin{theorem}[Local existence]\label{thm:localexistence}
If the data $\gf_0=\gf_0(s)$ exists and is analytic in $\CC_+$,
then for some $T>0$, there exists a unique solution $\gf=\gf(s;t)$ to the nonlinear
equation in Definition~\ref{def:LTeqn} for $t\in[0,T]$. While it exists, this solution 
is analytic in $s$ for $s\in\CC_+$ and smooth in time.
\end{theorem}
\begin{remark}
If we restrict the parameters in the general Smoluchowski-type equation, 
we can establish global in time solutions as in Example~\ref{example:explicitsolution}.
Suppose $B_0=\frac12$, $\beta=d_0=0$ and $a\equiv b_0\equiv0$.
Then, $\pf(s;t)=\exp\bigl(-tD_0-\int_0^t M(\tau)\,\rd\tau\bigr)\gf_0(s)$. 
Using the $M=M(t)$ form preceding Definition~\ref{def:LTeqn}, we see,
$\exp\bigl(-\int_0^t M(\tau)\,\rd\tau\bigr)=\bigl(1+(M_0/2D_0)(1-\mathrm{e}^{-tD_0})\bigr)^{-2}$.
Substituting this into \eqref{eq:pform} and \eqref{eq:qform}, by direct computation we find,
\begin{align*}
\pf(s;t)&=\mathrm{e}^{-tD_0}\gf_0(s)\biggl(1+\frac{M_0}{2D_0}(1-\mathrm{e}^{-tD_0})\biggr)^{-2},\\
\qf(s;t)&=\frac{\gf_0(s)}{M_0}\biggl(1+\frac{M_0}{2D_0}(1-\mathrm{e}^{-tD_0})\biggr)^{-1}-\frac{\gf_0(s)}{M_0}.
\end{align*}
Since $\gf=\pf/(1+\qf)$ we find, 
\begin{multline*}
\gf(s;t)=\mathrm{e}^{-tD_0}\gf_0(s)
\biggl(1+\frac{M_0}{2D_0}\bigl(1-\mathrm{e}^{-tD_0}\bigr)\biggr)^{-1}\\
\cdot\biggl(1+\frac{1}{2D_0}\bigl(M_0-\gf_0(s)\bigr)\bigl(1-\mathrm{e}^{-tD_0}\bigr)\biggr)^{-1}.
\end{multline*}
Since for all $s\in\CC_+$ we know $|\gf_0(s)|\leqslant|\gf_0(0)|\equiv M_0$, 
there is no blow-up for $\gf=\gf(s;t)$ for $t\in[0,\infty)$.
\end{remark}
With explicit knowledge of the Laplace transform solution $\gf=\gf(s;t)$ in hand, we now explore 
what we can say about the solution $g=g(x;t)$ to the original general Smoluchowski-type equation~\eqref{eq:origgenSmoltype}
given in Definition~\ref{def:gSmoluchowski}. We begin by considering the solution $\pf=\pf(s;t)$
to the first linear equation in Prescription~\ref{def:lineareqns}. 
Since $\gf_0=\gf_0(s)$ and $\mathfrak a(s;t)$ are analytic in $s$ for $s\in\CC_+$,
as is $s^n$ for $n\in\mathbb N$, the inverse Laplace transform $p=p(x;t)$ exists and is uniquely given for $x\in[0,\infty)$ by,
\begin{equation*}
  p(x;t)=\frac{1}{2\pi\mathrm{i}}\int_{-\mathrm{i}\cdot\infty}^{\mathrm{i}\cdot\infty}\mathrm{e}^{sx}\pf(s;t)\,\rd s,
\end{equation*}
where $\pf=\pf(s;t)$ is given by~\eqref{eq:pform}.
Using the Laplace transform Initial Value Theorem, see Theorem~\ref{thm:IVT} below, observe that 
the limit as $s\to\infty$ when $n\neq 2(2k-1)$, or the limit as $s\to\mathrm{i}\cdot\infty$ when $n=2(2k-1)$,
of `$s\,\pf(s;t)$' goes to zero for every $t\geqslant0$, and so $p(0^+;t)=0$ for all $t\geqslant0$.
Similar limits for $s^\ell\pf(s;t)$ tend to zero for all $\ell=2,3,\ldots,n$ and $t\geqslant0$,
so $\pa^\ell p(0^+;t)=0$ for all $\ell=1,2,\ldots,n-1$ and $t\geqslant0$. Hence $p=p(x;t)$
satisfies the boundary conditions.
Now observe, by differentiating the expression above for $p=p(x;t)$ with respect to time, 
\begin{equation*}
  \pa_tp=-D_0p-d_0\pa^np-Mp-a\conv p,
\end{equation*}
with $\pa=\pa_x$, $M=M(t)$ and $a=a(x;t)$. Hence the inverse Laplace transform of $\pf=\pf(s;t)$
satisfies the linearised form of the general Smoluchowski-type equation in Definition~\ref{def:gSmoluchowski}.
Now consider the inverse Laplace transform of $\qf=\qf(s;t)$ given by~\eqref{eq:qform}.
Note, since $\mathfrak b_0=\mathfrak b_0(s)$ is analytic in $s$ for $s\in\CC_+$, 
we know the inverse Laplace transform integral is well-defined from the properties we know for $\pf=\pf(s;t)$.
Since $m\leqslant n$, we see that,
\begin{equation*}
  q(x;t)=\int_0^t\bigl(b_0\conv p+\beta\pa_x^mp-B_0p\bigr)(x;\tau)\,\rd\tau. 
\end{equation*}
The inverse Laplace transforms $q=q(x;t)$ and $p=p(x;t)$ are thus well-defined, and satisfy the
linear equations we could have anticipated they would. Let us now focus on the linear relation
$\pf(s;t)=\gf(s;t)\bigl(1+\qf(s;t)\bigr)$. Given $q=q(x;t)$ and $p=p(x;t)$ as the unique
inverse Laplace transforms of $\qf=\qf(s;t)$ and $\pf=\pf(s;t)$, respectively, there is
thus a unique function $g=g(x;t)$ satisfying, $p(x;t)=g(x;t)+\int_0^xg(x-y;t)q(y;t)\,\rd y$,
for all $x\in[0,\infty)$ and $t\in[0,T]$, where the time $T>0$ is that in Theorem~\ref{thm:localexistence}.
This result follows from the results preceding the Local Existence Theorem~\ref{thm:localexistence},
including the Reciprocal Lemma~\ref{lemma:reciprocal}, which establish the Theorem.
We simply observe, $g=g(x;t)$ is the inverse transform of $\gf(s;t)=\pf(s;t)/\bigl(1+\qf(s;t)\bigr)$
which is well-defined, analytic in $s$ for $s\in\CC_+$ and smooth in time for $t\in[0,T]$ by
Lemma~\ref{lemma:reciprocal}. Further, from Lemma~\ref{lemma:reciprocal}, since for $t\in[0,T]$,
\begin{equation*}
\gf(s;t)=\pf(s;t)\bigl(1+\qf_*(s,t)\bigr),
\end{equation*}
with $\qf_\ast=\qf_\ast(s;t)$ analytic in $s$ for $s\in\CC_+$ 
and smooth in time for $t\in[0,T]$, the regularity of $g=g(x;t)$ is determined by the
regularity of $p=p(x;t)$ for $t\in[0,T]$. The regularity of $p=p(x;t)$ is determined
by its initial data $p_0\equiv g_0$. This follows, either from the form of the
exponential factor in the explicit expression for $\pf=\pf(s;t)$ given directly after Prescription~\ref{def:lineareqns},
or from the explicit partial differential equation $p=p(x;t)$ satisfies, shown just above.  
Indeed for the dissipative cases, the regularity of $p=p(x;t)$ for $t>0$ is strictly better than
that of the initial data. In any case, we assume $g_0$ and all its derivatives up
to and including that of order $n$ lie in $L^1\bigl(0,\infty;[0,\infty)\bigr)\cap C\bigl(0,\infty;[0,\infty)\bigr)$,
and that $g_0$ satisfies the zero boundary conditions at $x=0$, so that $\pa_x^\ell g_0(0)=0$, for all $\ell=0,1,\ldots,n-1$. 
The relation for $\gf$ in terms of $\pf$ and $\qf_\ast$ just above is equivalent to,
\begin{equation*}
  g(x;t)=p(x;t)+\int_0^xp(x-y;t)q_\ast(y;t)\,\rd y.
\end{equation*}
Taking the limit $x\to0$ we see that $g(0;t)=p(0;t)$ which equals zero.
Further, we can differentiate this linear convolution relation so that we have,
\begin{align*}
  \pa_xg(x;t)=&\;\pa_xp(x;t)+p(0;t)q_\ast(x;t)\\
  &\;+\int_0^x\pa_xp(x-y;t)q_\ast(y;t)\,\rd y.
\end{align*}
Taking the limit $x\to0$ we find that $\pa_xg(0;t)=\pa_xp(0;t)+p(0;t)q_\ast(x;t)$, which is also zero.
Taking further derivatives up to and including order `$n-1$' and taking the limit $x\to0$, reveals that
$\pa_x^\ell g(0;t)=\pa_x^\ell p(0;t)$ for all $\ell=0,1,2,\ldots,n-1$ and $t\in[0,T]$ and thus the
function $g=g(x;t)$, defined in this manner, satisfies all of the boundary conditions; as well as
the initial condition $g(x;0)=g_0(x)$. Together with the Local Existence Theorem~\ref{thm:localexistence},
the \emph{main result} of this section is the following.
\begin{corollary}[Explicit solution]
For some $T>0$, the solution $g=g(x;t)$ to the general Smolu-chowski-type equation given in Definition~\ref{def:gSmoluchowski}
is explicitly given by,
\begin{equation*}
  g(x;t)=\frac{1}{2\pi\mathrm{i}}\int_{-\mathrm{i}\cdot\infty}^{\mathrm{i}\cdot\infty}\mathrm{e}^{sx}\biggl(\frac{\pf(s;t)}{1+\qf(s;t)}\biggr)\,\rd s,
\end{equation*}
for $t\in[0,T]$, where $\pf=\pf(s;t)$ and $\qf=\qf(s;t)$ satisfy the linear equations given in Prescription~\ref{def:lineareqns} 
(with their explicit solution formulae stated directly afterwards). This solution satisfies the boundary conditions
stated in Definition~\ref{def:gSmoluchowski}. That it satisfies the general Smoluchowski-type equation 
in Definition~\ref{def:gSmoluchowski}, can be checked by direct substitution.
\end{corollary}

We record the following, which follows from Lemma~\ref{lemma:reciprocal}, as it is useful to us later.
\begin{lemma}[Inverse kernel]\label{lemma:IKL}
Let $q=q(x;t)$ denote the inverse Laplace transform of $\qf=\qf(s;t)$,
and suppose that $\qf(s;0)=0$ and that $\qf=\qf(s;t)$ is analytic in $s$ for $s\in\CC_+$ and smooth in time.
Then for some $T>0$, the function $\mathfrak r(s;t)\coloneqq(1+\qf(s;t))^{-1}$ exists for $t\in[0,T]$ and
is analytic in $s$ for $s\in\CC_+$ and smooth in time. 
The inverse Laplace transform $r=r(x;t)$ of $\mathfrak r=\mathfrak r(s;t)$ exists for $t\in[0,T]$,
and for any $f\in C(0,T;L^1([0,\infty);\R))$ we have,
\begin{equation*}
\int_0^x\biggl(f(y;t)+\int_0^yf(z;t)q(y-z;t)\,\rd z\biggr)\,r(x-y;t)\,\rd y
\end{equation*}
equals $f(x;t)$.
\end{lemma}
\begin{proof}
The statements preceding the final one, follow from Lemma~\ref{lemma:reciprocal}.
Taking the inverse Laplace transform of $\mathfrak f(s;t)\bigl(1+\qf(s;t)\bigr)\mathfrak r(s;t)=\mathfrak f(s;t)$,
which holds for the Laplace transform $\mathfrak f=\mathfrak f(s;t)$ of $f=f(x;t)$, generates the result. \qed
\end{proof}
In our arguments above we used the following generalisation of the Laplace transform Initial Value Theorem,
which relies on the generalised Riemann--Lebesgue Lemma. 
We give a proof in~\ref{app:IVTproof}.
\begin{theorem}[Initial Value Theorem]\label{thm:IVT}
Suppose $f,f'\in L^1_{\mathrm{loc}}\bigr([0,\infty);\R\bigr)$ are both of exponential order.
If $\mathfrak f=\mathfrak f(s)$ is the Laplace transform of $f$, then $f(0^+)=\lim_{s\to\infty}s\mathfrak f(s)$,
where the limit can be taken along any ray in the half plane $\mathrm{Re}(s)\geqslant0$, not including
the positive or negative imaginary $s$-axes. If we assume $f,f'\in L^1\bigr([0,\infty);\R\bigr)$,
then we can also take the limit $s\to\pm\infty\cdot\mathrm{i}$, i.e.\/ along the rays coinciding with the
positive or negative imaginary $s$-axes.
More generally suppose $f\colon[0,\infty)\to\R$, and all of its derivatives up to and including order $n\in\mathbb N$,
are locally integrable on $[0,\infty)$ and of exponential order. Then for $\ell\in\{0,1,\ldots,n-1\}$,
the $\ell$th order initial derivatives $f^{(\ell)}(0^+)$, are given by the limit as $s\to\infty$ of,
\begin{equation*}
    s^{\ell+1}\mathfrak f(s)-s^\ell f(0^-)-s^{\ell-1} f'(0^-)-\cdots-sf^{(\ell-1)}(0^-),
\end{equation*}
where the limit can be taken along any ray in the half plane as above, not including the imaginary axes.
If $f^{(\ell)}\in L^1\bigr([0,\infty);\R\bigr)$ for all $\ell\in\{0,1,\ldots,n\}$, then we can also take the 
limit along the positive or negative imaginary $s$-axes.
\end{theorem}

\section{Coarsening, depinning, multiple merger and other examples}\label{sec:examples}
We now consider a sequence of five examples, each of which does not fit exactly into the format
of the general Smoluchowski-type equation in Definition~\ref{def:gSmoluchowski}. However, in
each case, we can utilise the Grassmannian flow ideas developed above to either: explicitly solve them as
Examples~\ref{ex:coarseningmodel} and \ref{ex:DerridaRetaux}; develop a useful numerical approach as in Example~\ref{ex:Strangsplitting};
generalise our approach to coagulation systems with multiple mergers as in Example~\ref{ex:multiplemergers} or
make a connection to the viscous Burgers equation as in Example~\ref{ex:preLaplaceBurgers}.
\begin{example}[Coarsening]\label{ex:coarseningmodel}
Gallay \& Mielke~\cite{GM} analysed the following coarsening model for $g=g(x;t)$ for $x>t$ and with $g(x;t)=0$ for $x<t$:
\begin{equation*}
\pa_t g(x;t)=g(t;t)\int_0^{x-t}g(y;t)g(x-y-t;t)\,\rd y.
\end{equation*}
The initial data is prescribed at time $t=1$ so that $g(x;1)=g_1(x)$ for a suitable given function $g_1=g_1(x)$.
See Pego~\cite[p.~16--7]{Pego} for more details. Gallay and Mielke used a global linearisation transform to solve this model.
In our context, it is a special case of a multiplicative Grassmannian flow.
Taking the Laplace transform of the coarsening equation reveals that $\gf(s;t)=\int_t^\infty\mathrm{e}^{-sx} g(x;t)\,\rd x$
satisfies the Riccati equation, $\pa_t\gf=g(t,t)\,\mathrm{e}^{-st}\bigl(\gf^2-1\bigr)$,
with the initial condition $\gf(s;1)=\gf_1(s)$, where $\gf_1$ is the Laplace transform of $g_1$.
Now consider the linear prescription: 
\begin{equation*}
\pa_t\pf=c(s,t)\,\qf,\quad \pa_t\qf=b(s,t)\,\pf,\quad \pf=\gf\,\qf,
\end{equation*}
where $c$ and $b$ are given coefficient functions.
Note here, compared to the notation we used for the general Smoluchowski-type equation above,
we replace $1+\qf(s;t)$ by $\qf=\qf(s;t)$. We thus augment this prescription with
the initial conditions $\qf(s;1)=1$ and $\pf(s;1)=\gf_1(s)$.
A straightforward computation reveals $\gf=\gf(s;t)$ given in this prescription satisfies, $\pa_t\gf=c(s,t)-\gf\, b(s,t)\gf$.
This matches the Riccati equation for the Laplace transform of the coarsening
model above once we identify $c(s,t)=b(s,t)=-g(t;t)\,\mathrm{e}^{-st}$. For convenience we set,
$G(s;t)\coloneqq-\int_1^tg(\tau;\tau)\mathrm{e}^{-s\tau}\,\rd\tau$.
Note we have $G(s;1)=0$. Now we observe the two linear evolutionary equations
for $\pf$ and $\qf$ above, with the identifications for $b$ and $c$ indicated just above, satisfy, 
\begin{align*}
&&\pa_t\begin{pmatrix}\qf\\ \pf\end{pmatrix}
&=-g(t;t)\mathrm{e}^{-st}\begin{pmatrix}0&1\\1&0\end{pmatrix}\begin{pmatrix}\qf\\ \pf\end{pmatrix}\\
\Leftrightarrow&&\begin{pmatrix}\qf(s;t)\\ \pf(s;t)\end{pmatrix}
&=\exp\biggl(-G(s;t)\begin{pmatrix}0&1\\1&0\end{pmatrix}\biggr)\begin{pmatrix} 1 \\ \gf_1(s)\end{pmatrix},
\end{align*}
using the initial data for $\pf$ and $\qf$. 
Direct computation implies $\qf(s;t)=\cosh G(s;t)+\gf_1(s)\sinh G(s;t)$ and $\pf(s;t)=\sinh G(s;t)+\gf_1(s)\cosh G(s;t)$. 
Since $\gf=\pf/\qf$ we deduce, 
\begin{equation*}
\gf(s;t)=\frac{\tanh G(s;t)+\gf_1(s)}{1+\gf_1(s)\tanh G(s;t)}.
\end{equation*}
Further, since $s>0$: $\gf(s;t)\coloneqq\int_t^\infty\mathrm{e}^{-sx} g(x;t)\,\rd x\leqslant\mathrm{e}^{-st}\int_t^\infty g(x;t)\,\rd x=\mathrm{e}^{-st}M(t)$, 
where $M=M(t)$ is the total number of clusters at time $t$---as in Pego~\cite[p.~18]{Pego} we assume this is normalised to be one.
Since the right-hand side in this last inequality decays to zero as $t\to\infty$, we deduce, that
$\gf_1(s)=-\tanh G(s;\infty)=\tanh\int_1^\infty g(\tau;\tau)\mathrm{e}^{-st}\,\rd\tau$.
Hence by a standard identity for hyperbolic functions, $\gf(s;t)=\tanh\bigl(G(s;t)-G(s;\infty)\bigr)$, and thus,
\begin{equation*}
  \gf(s;t)=\tanh\int_t^\infty g(\tau;\tau)\,\rd\tau.
\end{equation*}
This matches the solution in Pego~\cite[eq.~(32)]{Pego}.
\end{example}
\begin{example}[Derrida--Retaux depinning]\label{ex:DerridaRetaux}
Derrida and Retaux~\cite{DerridaRetaux} consider the following depinning transition model for the field $r=r(x;t)$:
\begin{equation*}
\pa_t r=a_0\,\pa_xr+B_0\,r\ast r,
\end{equation*}
together with the condition $r(x;0)=r_0(x)$, for a given function $r_0=r_0(x)$.
This does not match the the general Smoluchowski-type equation in Definition~\ref{def:gSmoluchowski},
as $a_0$, which corresponds to `$-d_0$', is assumed to be strictly positive and we assume $B_0>0$ (only).
Derrida and Retaux look for underlying solutions to their depinning transition model of the form, $r(x;t)=R(t)\mathrm{e}^{-D(t)x}$,
where the functions $R=R(t)$ and $D=D(t)$ are determined as follows. Substituting this solution ansatz
into the depinning transition model above, we observe that necessarily $R=R(t)$ and $D=D(t)$ must satisfy,
$\dot{R}=-a_0DR$ and $\dot{D}=-B_0R$. 
Derrida and Retaux integrate this pair of ordinary differential equations to obtain explicit solutions,
revealing three different characteristic behaviours depending on whether: (i) $2R(0)B_0<a_0(D(0))^2$ which
corresponds to the unpinned phase; (ii) $2R(0)B_0=a_0(D(0))^2$ which corresponds to the critical case;
or (iii) $2R(0)B_0>a_0(D(0))^2$ which corresponds to the pinned phase. See Derrida and Retaux~\cite{DerridaRetaux}
for more details, in particular the critical time of divergence of $R=R(t)$ and $D=D(t)$ in case (iii).
With these in hand, our goal is to use the Grassmannian flow approach to derive closed form expressions
for more general solutions, indeed perturbative solutions. 
Suppose we define the perturbation $g=g(x;t)$ via the relation, $r(x;t)=R(t)\mathrm{e}^{-D(t)x}+g(x;t)$,
where $R=R(t)$ and $D=D(t)$ satisfy the ordinary differential equations above.
If $\varepsilon=\varepsilon(x;t)$ represents the function $\varepsilon(x;t)=\mathrm{e}^{-D(t)x}$,
then substituting this solution ansatz for $r=r(x;t)$ into the 
depinning transition model reveals that the perturbation $g=g(x;t)$ necessarily satisfies,
$\pa_t g=a_0\,\pa_xg+2B_0R\,\varepsilon\ast g+B_0\,g\ast g$. The Laplace transform of this is,
\begin{equation*}
\pa_t\gf=\biggl(a_0s+\frac{2RB_0}{s+D}\biggr)\gf+B_0\gf^2.
\end{equation*}
Consider the linear prescription, 
\begin{equation*}
\pa_t\pf=\biggl(a_0s+\frac{2RB_0}{s+D}\biggr)\pf,\quad \pa_t\qf=-B_0\pf,\quad\pf=\gf\qf.
\end{equation*}
Here, as in the last example, we replace $1+\qf(s;t)$ by $\qf=\qf(s;t)$. We augment this 
with the initial conditions $\qf(s;0)=1$ and $\pf(s;0)=\gf_0(s)$, where $\gf_0=\gf_0(s)$
is the Laplace transform of $g_0(x)=r_0(x)-R(0)\varepsilon(x;0)$. If we differentiate the linear relation $\pf=\gf\qf$
with respect to time using the product rule, we observe that $\gf=\gf(s;t)$ satisfies the Laplace transform
evolution equation, corresponding to the depinning transition model just above.
Let us now solve for $\qf=\qf(s;t)$ and $\pf=\pf(s;t)$. By direct integration we observe that, 
$\pf(s;t)=\exp\bigl(a_0st+2B_0\int_0^tR(\tau)/(s+D(\tau))\,\rd\tau\bigr)\gf_0(s)$,
while $\qf(s;t)=1-B_0\int_0^t\pf(s;\tau)\,\rd\tau$.
Hence the solution $\gf$ has the closed form $\gf(s;t)=\pf(s;t)/\qf(s;t)$,
which exists locally around $t=0$ due to the form of $\qf=\qf(s;t)$. We note, by direct computation
using the explicit forms for $\pf=\pf(s;t)$ and $\qf=\qf(s;t)$ above, for any finite small $t$ we have
$s\gf(s;t)\sim-(a_0/B_0)s^2$ as $s\to\infty$, so that within an infinitessimally small time interval,
$g=g(x;t)$ is divergent at the boundary $x=0$. Inserting the explicit forms for $R=R(t)$ and $D=D(t)$
into the solutions for $\pf=\pf(s;t)$ and $\qf=\qf(s;t)$ above, for each of the cases (i)---(iii),
and pursuing the analysis of the perturbative solution $g=g(x;t)$ to extend the analysis in
Derrida and Retaux~\cite{DerridaRetaux} is very much of interest, though it exceeds our goals here.
\end{example}
\begin{example}[Lambert--Schertzer genealogy]\label{ex:Strangsplitting}
The genealogy model of Lambert and Schertzer~\cite{LS} has the form, $\pa_t g=\pa_x(\alpha g)-D_0 g+\tfrac12 g\ast g$ for $g=g(x,t)$, 
where $\alpha=\alpha(x)$ has the specific form $\alpha=cx^2$, where $c$ is a constant, and $D_0=D_0(t)$ is a given function.
This model does not fit directly into the format of the general Smoluchowski-type equation in Definition~\ref{def:gSmoluchowski}
as the first linear vector field $V_1(g)\coloneqq\pa_x\bigl(\alpha g\bigr)$,
involves the non-constant coefficient term `$\alpha(x)\pa_xg$', and the second vector field $V_2(g)\coloneqq-D_0 g+\tfrac12 g\ast g$
involves the term `$-D_0g$' in which $D_0=D_0(t)$ is not constant.
However the vector field $V_1$ is integrable on its own, as is the second vector field $V_2$ which constitutes a Grassmannian flow.
This suggests we can construct a numerical Strang splitting method $\exp(\frac12\Delta tV_1)\circ\exp(\Delta tV_2)\circ\exp(\frac12\Delta tV_1)$ 
for which the individual flows are exact. Such a numerical method is second order in time
in general, and the spatial error is due to the discrete representation chosen.
Such knowledge, alongside classical schemes for such coagulation equations such as those of
Keck and Bortz~\cite{KB} or Carr \textit{et al.\/} \cite{CAD}, can thus be used to improve complexity.
\end{example}
\begin{example}[Multiple mergers]\label{ex:multiplemergers}
Smoluchowski models with multiple mergers arise naturally. The $n$-step descendent distribution of a Galton--Watson process
satisfies a discrete Smoluchowski flow, and the L\'evy jump measure of certain continuous state branching processes satisfy 
a multiple coalescence version of continuous Smoluchowski flow.
The spatial Laplace exponent $\gf=\gf(s;t)$ associated with the Lamperti transformation of the underlying
continuous state branching process satisfies,
\begin{equation}\label{eq:multiplemergers}
\pa_t\gf+\Psi(\gf)=0,
\end{equation}
with $\gf_0(s)=s$, where $\Psi$ is the given branching mechanism, see Iyer \textit{et al.\/} \cite{ILP}.
For example, for the Smoluchowski coagulation case with constant kernel $K=1$ we have considered hitherto,
the function $\Psi=\Psi(\gf)$ has the form $\Psi(\gf)=-\frac12\gf^2$; see for example Menon and Pego~\cite{MP}.
The general evolution equation~\eqref{eq:multiplemergers} can be solved by a general Grassmannian flow as follows.
Suppose the scalar functions $\pf=\pf(s;t)$, $\qf=\qf(s;t)$ and $\gf=\gf(s;t)$ satisfy the system of equations: 
\begin{equation*}
\pa_t\pf=0,\quad \pa_t\qf=\pf,\quad\pf=\Phi(\gf) \qf.
\end{equation*}
Here the nonlinear function $\Phi=\Phi(\gf)$ is determined as follows.
If we use the product rule to compute the time derivative of the third equation
we find, after dividing through by $\qf$, that $\gf$ necessarily satisfies $\pa_t\gf+\bigl(\Phi'(\gf)\bigr)^{-1}\Phi^2(\gf)=0$.
To match this equation to~\eqref{eq:multiplemergers} we require $\Phi=\Phi(\gf)$ to satisfy $\Phi'=(\Psi)^{-1}\Phi^2$.
This is a Riccati equation which is a Grassmannian flow. Indeed suppose the scalar fields
$P=P(\gf)$, $Q=Q(\gf)$ and $\Gamma=\Gamma(\gf)$ satisfy the linear system,
\begin{equation*}
\pa_{\gf} P=0,\quad \pa_{\gf} Q=-(\Psi)^{-1}P,\quad P=\Gamma Q.
\end{equation*}
Taking the derivative of the third equation with respect to $\gf$ and using
the product rule we observe that $\Gamma=\Gamma(\gf)$ satisfies the Riccati equation
$\pa_{\gf}\Gamma=(\Psi)^{-1}\Gamma^2$ matching the Riccati equation $\Phi'=(\Psi)^{-1}\Phi^2$
for $\Phi=\Phi(\gf)$. Hence we deduce that $\Phi=\Phi(\gf)$ is given by $\Phi=PQ^{-1}$,
where $P$ and $Q$ satisfy the linear equations shown. 
Since $P(\gf)=P_0$, a constant, we deduce $\Phi=P_0Q^{-1}$. Integrating the second equation we find,
\begin{align*}
&Q=Q_0-\int_{\gf_0}^{\gf}\bigl(\Psi(\tilde\gf)\bigr)^{-1}\,\rd\tilde\gf\,P_0\\
\Rightarrow~&\frac{1}{\Phi\bigl(\gf(s,t)\bigr)}-\frac{1}{\Phi\bigl(\gf_0(s)\bigr)}=-\int_{\gf_0(s)}^{\gf(s,t)}\frac{1}{\Psi(\tilde\gf)}\,\rd\tilde\gf.
\end{align*}
This suggests we take the defining relation for $\Phi$ to be $\bigl(\Phi(\gf)\bigr)^{-1}=-\int^{\gf}\bigl(\Psi(\tilde\gf)\bigr)^{-1}\,\rd\tilde\gf$,
i.e.\/ in terms of the indefinite integral shown, taking the additive constant to be zero.
Note that in principle we could have used this form for $\Phi=\Phi(\gf)$ in the relation between 
$\pf$ and $\qf$ from the outset. In any case the relation above demonstrates how we can recover the unknown
$\Phi=\Phi(\gf)$ from the known function $\Psi=\Psi(\gf)$.
Then the linear equations for $\pf=\pf(s;t)$ and $\qf=\qf(s;t)$ reveal $\pf(s,t)=\pf_0(s)$ and $\qf(s,t)=\qf_0(s)+\pf_0(s)t$
where $\pf_0=\pf_0(s)$ and $\qf_0=\qf_0(s)$ are functions of `$s$' only. We find,
$\Phi\bigl(\gf(s,t)\bigr)=\pf(s,t)/\qf(s,t)=\bigl(\bigl(\Phi(\gf_0(s))\bigr)^{-1}+t\bigr)^{-1}$.
Hence, provided we can compute the inverse function $\Phi^{-1}$, then $\gf=\gf(s,t)$ is given by, 
\begin{equation*}
\gf(s,t)=\Phi^{-1}\Biggl(\biggl(\Bigl(\Phi\bigl(\gf_0(s)\bigr)\Bigr)^{-1}+t\biggr)^{-1}\Biggr).
\end{equation*}
For example, consider the Smoluchowski case in Example~\ref{example:explicitsolution} with $\Psi(\gf)=-\frac12\gf^2$,
i.e.\/ assume we have subsumed the coagulation loss term `$-M\gf$' via the de-singularised Laplace transform.
Then the formula for $\Phi=\Phi(\gf)$ involving the indefinite integral above reveals
$\Phi(\gf)=-\frac12\gf$. The formula for $\gf=\gf(s,t)$ above reveals, $\gf(s,t)=2\gf_0(s)/(2-t\gf_0(s))$, 
in this special case. This matches the explicit solution given in Example~\ref{example:explicitsolution}
(ignoring the initial total number of clusters $M_0$). See Remark~\ref{rmk:Smoluchowskirigorous} and also Menon and Pego~\cite{MP}.
\end{example}
\begin{example}[Pre-Laplace viscous Burgers]\label{ex:preLaplaceBurgers}
For the general Smoluchowski-type equation in Definition~\ref{def:gSmoluchowski},
we imposed the condition that $d=d(\pa_x)$ and $b=b(\pa_x)$ have constant coefficients.
However there is at least one exceptional case as follows. Consider the following  equation for $g=g(x;t)$ given for $x\in[0,\infty)$ by,
$\pa_t g=\nu x^2 g+\tfrac12x\,g\ast g$,
where $\nu>0$ is a constant parameter. Suppose that initially $g(x;0)=g_0(x)$ for a given function $g_0=g_0(x)$.
This corresponds to the case where $D_0=-\nu x^2$ and $B_0=\frac12x$, with all other parameters zero, in the
general Smoluchowski-type equation. The $x$-dependence in $B_0$ and $D_0$ is outwith the conditions on $d$ and $b$ mentioned.
Taking the Laplace transform, we see that $\gf=\gf(s;t)$ satisfies, 
\begin{equation*}
\pa_t\gf=\nu\pa_s^2\gf+\tfrac12\pa_s(\gf^2),
\end{equation*}
with $\gf(s;0)=\gf_0(s)$, where $\gf_0=\gf_0(s)$ the Laplace transform of $g_0=g_0(x)$.
This is the \emph{viscous Burgers equation}. Consider the following prescription
of linear equations for the functions $\pf=\pf(s;t)$, $\qf=\qf(s;t)$ and $\gf=\gf(s;t)$:
\begin{equation*}
\pa_t\pf=\nu\pa_s^2\pf,\quad \qf=-\frac{1}{2\nu}\int_s^\infty\pf(\sigma;t)\,\rd\sigma,\quad \pf=\gf\qf.
\end{equation*}
Note here, as in Examples~\ref{ex:coarseningmodel} and \ref{ex:DerridaRetaux}, we replace $1+\qf(s;t)$ by $\qf=\qf(s;t)$.
We observe that with this precription, the function $\qf=\qf(s;t)$ also satisfies the heat equation $\pa_t\qf=\nu\pa_s^2\qf$,
and further, the definition of $\qf$ is equivalent to $\pf=2\nu\pa_s\qf$.
By direct computation, $\gf=\gf(s;t)$ can be shown to satisfy the viscous Burgers equation above provided $\qf(s;t)\neq0$.
We augment the prescription with the initial conditions $\pf(s;0)=\gf_0(s)$ and $\qf(s;0)=-(1/2\nu)\int_s^\infty\gf_0(\sigma)\,\rd\sigma$.
Hence, provided $\gf_0$ is not identically zero, then for a short time at least $\qf=\qf(s;t)$ will also
be non-zero. Indeed, the third linear, algebraic equation for $\gf$ implies $\gf=2\nu\pa_s\log\qf$,
i.e.\/ the \emph{Cole--Hopf transformation}. Hence a unique solution exists globally in time since $\qf=\qf(s;t)$
is a solution to the heat equation---say assuming initial data $g_0=g_0(x)$ for its Laplace transform $\gf_0=\gf_0(s)$
is analytic in $\CC_+$. Consequently, there is a unique global solution to the 
pre-Laplace equation above. This result is the analogue of that in Beck~\textit{et al.\/}~\cite[Remark~8]{BDMS1}.
\end{example}

\section{Additive and multiplicative kernel cases}\label{sec:FredholmGrassmannians}
We now consider Smoluchowski's equation~\eqref{eq:Smoluchowski} in the additive $K=x+y$ and multiplicative $K=xy$ coagulation
frequency kernel cases. These two cases are intimately related, see Deaconu and Tanr\'e~\cite[Section~3.3]{DT}.
Our goal in this section is to show that the additive and multiplicative kernel cases are also linearisable as
as Grassmannian flows, indeed they are rank-one analytic Grassmannian flows. To set the scene, we rely heavily on the
rigorous results established by Menon and Pego~\cite{MP} as follows.
For any positive Radon measure $\nu_t=\nu_t(\rd x)$ on $(0,\infty)$ with scalar real parameter $t$,
the \emph{desingularised} Laplace transform $\gf=\gf(s;t)$ is given by~\eqref{eq:desingLT}.
The \emph{modified-desingularised} Laplace transform $\gf^\ast=\gf^\ast(s;t)$ is defined by, 
\begin{equation}\label{eq:mdesingLT}
\gf^\ast(s;t)\coloneqq\int_0^\infty (1-\mathrm{e}^{-sx})x\,\nu_t(\rd x).
\end{equation}
The former is the natural context for the constant and additive cases, while the latter is 
the natural context for the multiplicative case.
In all cases, we assume the data $\nu_0$ is a positive Radon measure on $(0,\infty)$.
In the additive case, the desingularised Laplace transform $\gf=\gf(s;t)$ of the solution to Smoluchowski's equation~\eqref{eq:Smoluchowski}
satisfies the inviscid Burgers equation with linear decay ($\pa=\pa_s$),
\begin{equation}\label{eq:inviscidBurgersdamping}
\pa_t\gf=\gf\pa\gf-\gf.
\end{equation}
Let $\eta=\eta(s)$ denote the initial data corresponding to the desingularised Laplace transform of $\nu_0$, so $\gf(s;0)=\eta(s)$.  
Thus $\eta$ is analytic in $\CC_+$. Here we also used that in this case
the total mass of clusters $M_t(t)\coloneqq\int_{[0,\infty)}x\,\nu_t(\rd x)$ is constant, and we can normalise the initial data so that $M_1(t)\equiv1$.
Further, by definition, $\gf(0;t)\equiv 0$ and $\pa_s\gf(0;t)\equiv M_1(t)$; in particular $\eta(0)=0$ and $\pa_s\eta(0)=1$.
In the multiplicative case, the modified-desingularised Laplace transform of the solution to Smoluchowski's equation~\eqref{eq:Smoluchowski}
satisfies the inviscid Burgers equation,
\begin{equation}\label{eq:inviscidBurgers}
\pa_t\gf^\ast=\gf^\ast\pa\gf^\ast.
\end{equation}
Let $\eta^\ast=\eta^\ast(s)$ denote the initial data, analytic in $\CC_+$, such that $\gf^\ast(s;0)=\eta^\ast(s)$.
We observe that, by definition, $\gf^\ast(0;t)\equiv 0$, while $\pa_s\gf^\ast(0;t)\equiv M_2(t)$ where $M_2(t)\coloneqq\int_{[0,\infty)}x^2\,\nu_t(\rd x)$.
This means in particular that $\eta^\ast(0)=0$ and we can normalise the initial data so that $\pa_s\eta^\ast(0)=1$.
See Menon and Pego~\cite{MP} for more details. For the constant and additive kernel cases,
they establish there exists a unique positive Radon measure solution $\nu_t$ for all $t\in[0,\infty)$,
such that in the former case $t\mapsto\nu_t$, and in the latter case $t\mapsto x\nu_t$, is weakly continuous.
In the multiplicative kernel case, a solution exists on $[0,1)$ with $t\mapsto x^2\nu_t$ weakly continuous.
We henceforth focus on the classical problem of determining analytical solutions
to the nonlinear equations~\eqref{eq:inviscidBurgersdamping} and \eqref{eq:inviscidBurgers},
with data satisfying the properties mentioned, and demonstrating that they are linearisable as rank-one analytic Grassmannian flows.

Our approach is inspired by Byrnes~\cite{Byrnes} and Byrnes and Jhemi~\cite{BJ} and their
approach to an optimal nonlinear control problem and Riccati partial differential equation for the current state feedback.
We adapt their approach, which applies in a multi-dimensional context, to our scalar and one-dimensional application.
For convenience we adopt the notation $\mathfrak f_t\circ a$ to denote functions $\mathfrak f=\mathfrak f(a;t)$ and often
assume the time-dependence is implicit and write $\mathfrak f\circ a$. Consider the following prescription.
\begin{prescription}[Invisc. Burgers]\label{prescription:inviscidBurgers}
For $a\in\CC_+$, suppose $\qf=\qf_t$, $\pf=\pf_t$ and $\gf=\gf_t$
satisfy $\qf_0\circ a=a$ and $\pf_0\circ a=\eta\circ a$ and the system of equations,
\begin{equation*}
\pa_t\pf=0,\qquad 
\pa_t\qf=-\pf,\qquad
\pf=\gf\circ\qf.
\end{equation*}
We assume the initial data $\eta=\eta\circ a$ is analytic on $\CC_+$ with $\eta\circ0=0$.
For the additive case, we replace the first equation by $\pa_t\pf=-\pf$.
\end{prescription}
For the moment, assume all the relevant fields are analytic in $a\in\CC_+$ and smooth in time---we discuss the validity of this presently.
Then it is straightforward to show that $\gf=\gf_t\circ\qf$ satisfies the inviscid Burgers equation~\eqref{eq:inviscidBurgers},
and with the adaptation for the additive case, $\gf=\gf_t\circ\qf$ would satisfy the inviscid Burgers equation~\eqref{eq:inviscidBurgersdamping}.
Equation~\eqref{eq:inviscidBurgersdamping} or \eqref{eq:inviscidBurgers}, can be solved by characteristics
and in either case, the linear equations for $\pf$ and $\qf$ prescribe the characteristic coordinates $\qf=\qf_t\circ a$,
with the characteristics labelled by the coordinate $a\in\CC_+$ at time $t=0$.
For the adapted equations in Prescription~\ref{prescription:inviscidBurgers} with $\pa_t\pf=-\pf$ $\Leftrightarrow$ $\pf_t=\mathrm{e}^{-t}\eta$,
which corresponds to the additive kernel case, the characteristic coordinates are given by $\qf_t\circ a=a-(1-\mathrm{e}^{-t})\,\eta\circ a$.
The solution $\gf=\gf_t\circ\qf$ thus has the form,
\begin{equation}\label{eq:Burgersdampedsolution}
\gf_t\circ\qf=\mathrm{e}^{-t}\eta\circ\bigl(\id-(1-\mathrm{e}^{-t})\eta\bigr)^{-1}\circ q.
\end{equation}
For the equations given in Prescription~\ref{prescription:inviscidBurgers}, corresponding to the multiplicative case, $\pf_t=\eta$
and the characteristic coordinates are $\qf_t\circ a=a-t\,\eta\circ a$. The solution $\gf=\gf_t\circ\qf$ has the form,
\begin{equation}\label{eq:Burgerssolution}
\gf_t\circ\qf=\eta\circ\bigl(\id-t\eta)^{-1}\circ q.
\end{equation}
\begin{remark}\label{rmk:nocoincidence}
  It is no coincidence that the solution $\gf=\gf_t$ to the Riccati equation $\pa_t\gf=\gf^2$ can be generated
  from Prescription~\ref{prescription:inviscidBurgers} as follows---recall Prescription~\ref{prescription:constantkernel}
  from Section~\ref{sec:Smoluchowski}. If we replace the function composition `$\circ$' everywhere in
  Prescription~\ref{prescription:constantkernel} by scalar multiplication so that, for example, $\pf_0\circ a$  
  becomes $\pf_0\,a$, then the initial data becomes $\qf=1$ and $\pf_0=\eta$. The linear evolution equations for
  $\pf$ and $\gf$ generate $\pf=\eta$ and $\qf=1-t\eta$. The third relation reduces to $\pf=\gf\,\qf$, from
  which we deduce $\gf=\gf_t$ satisfies $\pa_t\gf=\gf^2$ which has the solution,
  \begin{equation*}
  \gf_t=\eta\,(1-t\eta)^{-1}.
  \end{equation*}
  This naturally suggests, the generalisation to the compositional dependence $\pf=\gf\circ\qf$ in Prescription~\ref{prescription:inviscidBurgers},
  generates the consequential generalisation from the Riccati to the inviscid Burgers equation; as Byrnes~\cite{Byrnes} and Byrnes and Jhemi~\cite{BJ} concluded.
\end{remark}
We can push Remark~\ref{rmk:nocoincidence} further.
Formally, we can construct a \emph{nonlinear graph manifold} as follows. Given fields $\qf$ and $\pf$ the
relation $\pf=\gf\circ\qf$ prescribes a nonlinear graph $\gf$, which represents a coordinate patch on such a manifold.
Indeed we have,
\begin{equation}\label{eq:gnonlinear}
\begin{pmatrix} \qf \\ \pf \end{pmatrix}
=\begin{pmatrix} \qf \\ \gf\circ\qf \end{pmatrix}
=\begin{pmatrix} \id \\ \gf\end{pmatrix}\circ\qf. 
\end{equation}
Equivalencing by all compositions with $\qf$, the form $(\id,\gf)$ represents the graph of the function $\gf$. 
We could also consider the map $\gf^\prime\colon\pf\mapsto\qf$ so that,
\begin{equation}\label{eq:gpnonlinear}
\begin{pmatrix} \qf \\ \pf \end{pmatrix}
=\begin{pmatrix} \gf^\prime\circ\pf \\ \pf \end{pmatrix}
=\begin{pmatrix} \gf^\prime \\ \id\end{pmatrix}\circ\pf. 
\end{equation}
Equivalencing by all compositions with $\pf$, generates the form $(\gf^\prime,\id)$, respresenting another coordinate patch
of the nonlinear graph manifold. Naturally, we have $\gf^\prime=\gf^{-1}$. In this case if $\pf$ and $\qf$ satisfy the
linear evolution equations in Prescription~\ref{prescription:inviscidBurgers}, then,
\begin{equation*}
\eta=\pf_t=-\pa_t\qf=-(\pa_t\gf_t^\prime)\circ\pf_t=-(\pa_t\gf_t^\prime)\circ\eta.
\end{equation*}
In other words $\gf^\prime=\eta^{-1}-t\,\id$, as we might expect; in particular, compare this with \eqref{eq:Burgerssolution}.

Our goal now is to show that we can pullback this notion of a nonlinear graph manifold
to the rigorous setting of a rank-one analytic Grassmann manifold.
The key to this is the Fa\`a di Bruno formula for the composition of analytic functions and the Lagrange Inversion Theorem.
For the moment, we closely follow Figueroa \textit{et al.\/} \cite{FG-BV}. Consider the set of exponential power series of the form,
\begin{equation*}
\ff\circ a=\sum_{n\geqslant1}\frac{\ff_n}{n!}a^n,
\end{equation*}
with $\ff_1\neq0$. We can represent such series by the vector of their coefficients, $\hat\ff=(\ff_1,\ff_2,\ff_3,\cdots)^{\mathrm{T}}$. 
We can compute the composition $\hf=\ff\circ\qf$ of two such series in terms of their
coefficients $\hat\qf$ and $\hat\ff$. At the exponential power series level we have,
\begin{equation*}
\hf\circ a=\sum_{n\geqslant1}\frac{\ff_n}{n!}\Biggl(\sum_{k\geqslant1}\frac{\qf_k}{k!}a^k\Biggr)^n.
\end{equation*}
Cauchy's product or Fa\`a di Bruno formula reveals,
\begin{equation}\label{eq:FaadiBruno}
\hf_n=\sum_{k=1}^n\Biggl(\sum_{\bs{i}\in\mathcal C(n,k)}\frac{n!\,\qf_{i_1}\cdots\qf_{i_k}}{i_1!\cdots i_k!}\Biggr)\frac{\ff_k}{k!},
\end{equation}
where the sum is over all $\bs{i}=(i_1,\ldots,i_k)\in\mathcal C(n,k)$; the set of all compositions of $n$ of length $k$.  
Since all quantities here are scalars, we could use partitions instead; see Figueroa \textit{et al.\/} \cite{FG-BV}.
Naturally, the composition $\ff\circ\qf$ of any two such exponential power series is analytic in the domain of $\qf$;
again see Figueroa \textit{et al.\/} \cite{FG-BV} (a minor adaptation is required).
Hence an inverse exponential power series $\ff$ such that $\ff\circ\qf=\id$ exists.
The Lagrange Inversion Theorem gives the coefficients $\hat\ff$ of the inverse $\ff$. See Gessel~\cite{Gessel}.
The set of such exponential power series thus forms a group.

We now give a (infinite) matrix representation for such exponential power series. Henceforth we
restrict ourselves to exponential series $\ff$ for which $\ff_1=1$.
Let $\ell_{\mathrm{a}}$ denote the set of sequences of coefficients $\hat\ff^1=[1,\ff_2,\ff_3,\cdots]^{\mathrm{T}}$
associated with exponential power series analytic in $\CC_+$. We regard $\ell_{\mathrm{a}}$ as a chart in the
projective space of sequences: we equivalence all sequences $(\ff_1,\ff_2,\ff_3,\cdots)^{\mathrm{T}}$ generated
from exponential power series to $[1,\ff_2,\ff_3,\cdots]^{\mathrm{T}}$ by simply dividing through by $\ff_1\neq0$.
That the first element is unity distinguishes the chart. 
The natural product $\hat\ff^1\cdot\hat\qf^1$ on $\ell_{\mathrm{a}}$ is that generated
by the Fa\`a di Bruno formula~\eqref{eq:FaadiBruno} with unit $\hat\ef=[1,0,0,0,\ldots]^{\mathrm{T}}$.
There is a natural group action of $\ell_{\mathrm{a}}$ on itself, however, there is also a group action on $\ell_{\mathrm{a}}$,
given by (infinite) matricies $Q=Q(\hat\qf^1)$ of lower triangular form with $1$'s down the diagonal, i.e.\/ of the form, 
\begin{equation*}
  Q=\begin{pmatrix}
  1      & 0      & 0      & \cdots \\
  Q_{21} &  1     & 0      & \cdots \\
  Q_{31} & Q_{32}  & 1      & \cdots \\
  \vdots & \vdots & \vdots & \ddots   
    \end{pmatrix},
\end{equation*}
where we assume the entries have the form,
\begin{equation}\label{eq:Qform}
  Q_{nk}\coloneqq\sum_{\bs{i}\in\mathcal C(n,k)}\frac{n!\,\qf_{i_1}\cdots\qf_{i_k}}{k!i_1!\cdots i_k!},
\end{equation}
for any $\hat\qf^1\in\ell_{\mathrm{a}}$. The $Q_{nk}$ have the form of the coefficients of $\ff_k$
in the Fa\`a di Bruno formula for $\hf=\ff\circ\qf$ in~\eqref{eq:FaadiBruno}.
We know from our results above that such matrices form a group, which we denote by $\mathrm{SL}(\ell_{\mathrm{a}})$,
and that $\mathrm{SL}(\ell_{\mathrm{a}})$ generates a natural (pre-multiplicative) left group action on $\ell_{\mathrm{a}}$.
For example, we know that there exists a $Q=Q(\hat\qf^1)$ such that $\hat\hf^1=Q\,\hat\ff^1$ for any pair $\hat\ff^1,\hat\hf^1\in\ell_{\mathrm{a}}$.
Further, for any $Q\in\mathrm{SL}(\ell_{\mathrm{a}})$, $\det Q=1$.
Naturally, the coefficients of the inverse in the Lagrange Inversion Theorem generates the entries in $Q^{-1}$.
Lastly, note that if $Q=Q(\hat\qf^1)$ is generated from $\hat\qf^1$ with entries as in~\eqref{eq:Qform}, then, $Q^{-1}\,\hat\qf^1=\hat\ef$. 

To construct our Grassmannian we need a vector space of sequences. Let $\ell_0$ denote the vector space of sequences of coefficients
$\hat\ff^0=[0,\ff_2,\ff_3,\cdots]^{\mathrm{T}}$ associated with exponential power series analytic in $\CC_+$;
i.e.\/ sequences for which $\ff_1=0$. Such a set of sequences is closed under addition, which here is element-wise addition through the sequences,
and the zero element is $\of=[0,0,0,\ldots]$ as $\hat\ff^0+\of=\hat\ff^0=\of+\hat\ff^0$. 
Using $\ell_0$, we can construct a vector space $\mathbb V\coloneqq\ell_{\mathrm{a}}\oplus\ell_0$ as follows.
We use the convention that when we add projective elements from $\ell_{\mathrm{a}}$, say $\hat\ff^1+\hat\gf^1$, we always
equivalence so the resulting sequence has first element equal to `$1$', and this lies in $\ell_{\mathrm{a}}$.
We observe that $\hat\ff^1+\hat\gf^0\in\ell_{\mathrm{a}}$ for any $\hat\ff^1\in\ell_{\mathrm{a}}$ and $\hat\gf^0\in\ell_0$.
All the vector space axioms for $\mathbb V$ are easily checked, noting that scalar multiplication rules
are subject to the convention mentioned. Note that $\mathrm{SL}(\ell_{\mathrm{a}})$ acts on $\ell_0$, though $Q\,\of=\of$
for all $Q=Q(\hat\qf^1)$. We now construct the rank-one analytic Grassmannian $\mathrm{Gr}(\Vb^{\times2};\Vb)$.
Suppose $\hat\qf^\ast,\hat\pf^\ast\in\Vb$, where $\hat\qf^\ast$ could be $\hat\qf^1$ or $\hat\qf^0$,
the same sequence but starting respectively with `$1$' or `$0$'.
Likewise for $\hat\pf^\ast$. Let,
\begin{equation*}
\lf\coloneqq\begin{pmatrix} \hat\qf^\ast \\ \hat\pf^\ast \end{pmatrix}\qquad\text{and}\qquad
\lf_0\coloneqq\begin{pmatrix} \hat\ef \\ \of \end{pmatrix},
\end{equation*}
respectively, denote a generic line in the two-dimen-sional sequence space $\Vb^{\times2}$,
and the fixed line parallel to the first component axis in $\Vb^{\times2}$. Naturally
we have $\lf_0^\perp=(\of,\hat\ef)^{\mathrm{T}}$, and the projections of $\lf$ onto $\lf_0$ and $\lf_0^\perp$ are, respectively,
\begin{equation*}
\lf^\parallel=\begin{pmatrix} \hat\qf^\ast \\ \of \end{pmatrix}\qquad\text{and}\qquad
\lf^\perp=\begin{pmatrix} \of \\ \hat\pf^\ast \end{pmatrix}.
\end{equation*}
The first projection is achievable iff $\hat\qf^\ast\neq\hat\qf^0$, while the second is achievable iff $\hat\pf^\ast\neq\hat\pf^0$.
Assuming for the moment, $\hat\qf^\ast=\hat\qf^1$, then we observe the line $\lf^\parallel$ coincides with the line $\lf_0$.
If $\hat\pf^\ast=\hat\pf^1$, then the line $\lf^\perp$ coincides with $\lf_0^\perp$.
Indeed the transformations $Q^{-1}=Q^{-1}(\hat\qf^1)$ and $Q^{-1}=Q^{-1}(\hat\pf^1)$ in $\mathrm{SL}(\ell_{\mathrm{a}})$
respectively, transform $\hat\qf^1$ to $\hat\ef$, and $\hat\pf^1$ to $\hat\ef$.
Under these transformations, $\lf$ respectively becomes,
\begin{equation*}
\begin{pmatrix} \hat\ef \\ \hat\gf \end{pmatrix}\qquad\text{or}\qquad \begin{pmatrix} \hat\gf^\prime \\ \hat\ef \end{pmatrix},
\end{equation*}
where $\hat\gf=Q^{-1}(\hat\qf^1)\,\hat\pf^\ast$ and $\hat\gf^\prime=Q^{-1}(\hat\pf^1)\,\hat\qf^\ast$. Compare these with
\eqref{eq:gnonlinear} and \eqref{eq:gpnonlinear}. The first form above parameterises all lines that can be projected onto
the line $\lf_0$, while the second parameterises all those that can be projected onto $\lf_0^\perp$. The two sets intersect.
Indeed only $\lf_0^\perp$ cannot be represented by the first form, and only $\lf_0$ cannot be represented by the second form.
The forms represent coordinate charts of the rank-one analytic Grassmannian $\mathrm{Gr}(\Vb^{\times2};\Vb)$ or the Grassmannian
of lines in two-dimensional sequence space. The charts are compatible as, in the intersection, necessarily $\hat\qf^\ast=\hat\qf^1$
and $\hat\pf^\ast=\hat\qf^1$. Thus given the line $(\hat\ef,\hat\ff^1)^{\mathrm{T}}$ in the first chart, the transformation
$Q^{-1}=Q^{-1}(\hat\ff^1)$ generates,
\begin{equation*}
\begin{pmatrix} Q^{-1}(\hat\ff^1)\,\hat\ef\\ Q^{-1}(\hat\ff^1)\,\hat\ff^1\end{pmatrix}\equiv\begin{pmatrix}(\hat\ff^1)^{-1}\\\hat\ef\end{pmatrix},
\end{equation*}
in the second chart, and vice-versa via an analogous transformation. Note that the first column of $Q^{-1}=Q^{-1}(\hat\ff^1)$
corresponds to $(\hat\ff^1)^{-1}$. Our prescription of $\mathrm{Gr}(\Vb^{\times2};\Vb)$ is thus complete.

We now return to the inviscid Burgers equation, Prescription~\ref{prescription:inviscidBurgers} and solutions
\eqref{eq:Burgersdampedsolution} and \eqref{eq:Burgerssolution}---the former corresponding to the damped version of the equation.
We can simplify the problem for ourselves by noting the following aspect of the characteristics solution or equivalently
Prescription~\ref{prescription:inviscidBurgers}. For either of the solutions \eqref{eq:Burgersdampedsolution} or \eqref{eq:Burgerssolution},
the real work is in computing the inverse of $\qf_t=\qf_t\circ a$. The leftmost function compositions, respectively
$\mathrm{e}^{-t}\eta$ and $\eta$, just represent \emph{observation} functions. They attach passive values that are transported by the flow
generated by the characteristics $\qf_t=\qf_t\circ a$. We can apply such observation functions as a second stage.
In other words, we could simplify Prescription~\ref{prescription:inviscidBurgers} to $\pa_t\pf=0$, or $\pa_t\pf=-\pf$ in the additive case,
$\pa_t\qf=-p$, and $\id=\ff_t\circ\qf_t$. So the goal is to compute $\ff_t$, the compositional inverse to $\qf_t$.
Recall we assume the data $\eta=\eta\circ a$ is analytic for $a\in\CC_+$, and that for both 
cases $\eta\circ0=0$ and $\pa\eta\circ0=1$. Hence in both cases $\eta$ has an exponential power series expansion with
coefficients $\hat\eta^1\in\ell_{\mathrm{a}}\subset\Vb$.

First, for the additive case, we know $\pf_t=\mathrm{e}^{-t}\eta$ and thus $\qf_t\circ a=a-(1-\mathrm{e}^{-t})\,\eta\circ a$.
The coefficients $\pf_1=\pf_1(t)$ and $\qf_1=\qf_1(t)$ in the exponential power series expansions of $\pf_t$ and $\qf_t$ are given by $\pf_1=\qf_1=\mathrm{e}^{-t}$,
which is non-zero for all $t\in[0,\infty)$.
Hence $\qf_t$ and $\pf_t$ can be respectively represented by $\hat\pf^1=\hat\pf^1(t)$ and $\hat\qf^1=\hat\qf^1(t)$ in $\ell_{\mathrm{a}}$ for all $t\in[0,\infty)$. 
Second, for the multiplicative case, we know $\pf_t=\eta$ and $\qf_t\circ a=a-t\,\eta\circ a$.
The corresponding exponential power series expansion coefficients $\pf_1=\pf_1(t)$ and $\qf_1=\qf_1(t)$ are given respectively by $\pf_1=1$ and $\qf_1=1-t$.
The latter coefficient is non-zero for $t\in[0,1)$, limiting our analysis for this case to these times. 
Thus $\qf_t$ and $\pf_t$ can be respectively represented by $\hat\pf^1=\hat\pf^1(t)$ and $\hat\qf^1=\hat\qf^1(t)$ in $\ell_{\mathrm{a}}$ for $t\in[0,1)$. 
Then in either case, the relation $\id=\ff_t\circ\qf_t$ becomes the relation $\hat\ef=\hat\ff^1\cdot\hat\qf^1$,
where the product is that generated by the Fa\`a di Bruno formula.
This is equivalent to the relation,
\begin{equation}\label{eq:crucialrelation}
\hat\ef=Q(\hat\qf^1)\,\hat\ff^1.
\end{equation}
We know the solution is $\hat\ff^1=Q^{-1}(\hat\qf^1)\,\hat\ef=(\hat\qf^1)^{-1}$.
Our stated goal is thus complete. The crucial fact at the end of the day, is that we can couch the 
additive and multiplicative kernel cases as Grassmannian flows and they are linearisable in this sense.
Their solution can be achieved by solving the linear equations for the fields $\pf_t$ and $\qf_t$ with
exponential power series representations, and then we can determine the solution by solving the (infinite)
system linear algebraic equations~\eqref{eq:crucialrelation} for $\hat\ff^1$.

\section{Discussion}\label{sec:discussion}
Future directions to pursue include generalising our Grassmannian flow analysis in Section~\ref{sec:FredholmGrassmannians}
to multi-dimensional and non-commutative scenarios and utilising composition operators; see Cowen and MacCluer~\cite{CM}, Kucik~\cite{Kucik}.
Further, Smoluchowski coagulation models are underpinned by stochastic processes, see Aldous~\cite{Aldous}.
For example, the nested Kingman coalescent underlies the genealogy model, see Lambert and Schertzer~\cite{LS}, while
a McKean--Vlasov process underlies the Derrida--Retaux model, see Hu \textit{et al.\/} \cite{HMP}. 
Connecting our analysis to such processes and looking towards more general kernels that arise in applications
is very much of interest, as is developing efficient Monte Carlo methods to simulate such phenomena.





\section{Declarations}

\subsection{Acknowledgement}
We thank Chris Eilbeck and Mohammed Kbiri Alaoui for very useful discussions.

\subsection{Funding and conflicts or competing interests}
SJAM was supported by an EPSRC Mathematical Sciences Small Grant EP/X018784/1. IS was supported by an EPSRC DTA Scholarship. 
There are no conflicts of interests or competing interests. 

\subsection{Data availability statement}
No data was used in this work.

\appendix

\section{Initial Value Theorem proof}\label{app:IVTproof}
For completeness, we prove Theorem~\ref{thm:IVT}. It relies on the generalised Riemann--Lebesgue Lemma.
We follow Lundberg \textit{et al.\/} \cite{LMT:IVT}. 
We first prove the basic case. For any $f\colon[0,\infty)\to\R$ with
$f,f'\in L^1_{\mathrm{loc}}\bigr([0,\infty);\R\bigr)$ and both of exponential order,
or $f,f'\in L^1\bigr([0,\infty);\R\bigr)$, in $\CC_+$:
$\int_{0^-}^\infty \mathrm{e}^{-sx}f'(x)\,\rd x=-f(0^-)+s\int_{0^-}^\infty \mathrm{e}^{-sx}f(x)\,\rd x$.
Hence in $\CC_+$:
$s\mathfrak f(s)=\int_{0^-}^\infty \mathrm{e}^{-sx}f'(x)\,\rd x+f(0^-)=f(0^+)+\int_{0^+}^\infty \mathrm{e}^{-sx}f'(x)\,\rd x$.
We take $s\to\infty$ with $\mathrm{Re}(s)\geqslant0$ on both sides.
Set $s=R\mathrm{e}^{\mathrm{i}\theta}$ with $\theta\in[-\pi/2,\pi/2]$. First suppose
$\theta\in(-\pi/2,\pi/2)$. Then focusing on the limit for the integral on the
right-hand side, we observe,  $\lim_{s\to\infty}\int_{0^+}^\infty \mathrm{e}^{-sx}f'(x)\,\rd x$
is bounded by $\lim_{R\to\infty}\int_{0^+}^\infty \mathrm{e}^{-Rx\cos\theta}\,|f'(x)|\,\rd x$,
which converges to zero by the Dominated Convergence Theorem since $\mathrm{e}^{-Rx\cos\theta}\to 0$ as $R\to\infty$,
pointwise in $x$ for $x>0$ and $\theta\in(-\pi/2,\pi/2)$. Now suppose $\theta=\pm\pi/2$. 
Then, $\lim_{s\to\infty}\int_{0^+}^\infty \mathrm{e}^{-sx}f'(x)\,\rd x=\lim_{R\to\infty}\int_{0^+}^\infty \mathrm{e}^{\pm\mathrm{i}Rx}f'(x)\,\rd x$,
which converges to zero by the Riemann--Lebesgue Lemma since in this case we assume $f'\in L^1\bigr([0,\infty);\R\bigr)$.
This gives the basic result.
Now suppose for all $\ell\in\{0,1,\ldots,n\}$, $f^{(\ell)}\in L^1_{\mathrm{loc}}\bigr([0,\infty);\R\bigr)$ and of exponential order,
or $f^{(\ell)}\in L^1\bigr([0,\infty);\R\bigr)$. In either case, 
for $\mathrm{Re}(s)\geqslant0$ and any $\ell\in\{1,2\ldots,n\}$ we know
$\int_{0^-}^\infty \mathrm{e}^{-sx}f^{(\ell)}(x)\,\rd x=-f^{(\ell-1)}(0^-)+s\int_{0^-}^\infty \mathrm{e}^{-sx}f^{(\ell-1)}(x)\,\rd x$,
i.e., 
\begin{multline*}
s\biggl(s^{\ell-1}\mathfrak f(s)-\sum_{l=2}^{\ell}s^{\ell-l}f^{(l-2)}(0^-)\biggr)\\
=\int_{0^-}^\infty \mathrm{e}^{-sx}f^{(\ell)}(x)\,\rd x+f^{(\ell-1)}(0^-).
\end{multline*}   
Now, from the same juncture, repeating the arguments for the basic case, but with $f'$ replaced by $f^{(\ell)}$,
establishes the more general result.

\end{document}